\documentclass[10pt,draft,reqno]{amsart}
     \makeatletter
     \def\section{\@startsection{section}{1}%
     \z@{.7\linespacing\@plus\linespacing}{.5\linespacing}%
     {\bfseries
     \centering
     }}
     \def\@secnumfont{\bfseries}
     \makeatother
\newtheorem{theorem}{Theorem}[section]

\newtheorem{proposition}[theorem]{Proposition}

\theoremstyle{definition}

\theoremstyle{remark}
\newtheorem{remark}[theorem]{Remark}
\numberwithin{equation}{section}
\begin{document}

\title[short title for running heading]{An extension of bifractional
Brownian motion}

\author{Xavier Bardina*}
\address{Xavier Bardina: Departament de Matem\`atiques, Facultat de Ci\`encies, Universitat Aut\`onoma de Barcelona, 08193-Bellaterra, Barcelona, Spain}
\email{bardina@mat.uab.cat}
\urladdr{http://mat.uab.cat/$\sim$bardina}

\author[Khalifa Es-Sebaiy]{Khalifa Es-Sebaiy}
\thanks{* This author is supported by the grant MTM2009-08869 from the Ministerio de Ciencia e Innovaci\'on}
\address{Khalifa Es-Sebaiy: SAMM, Centre d'Economie de La Sorbonne,
Universit\'e Paris 1 Panth\'eon-Sorbonne,
90, rue de Tolbiac, 75634, Paris, France}
\email{Khalifa.Es-Sebaiy@univ-paris1.fr}



\begin{abstract}
In this paper we introduce and study   a self-similar Gaussian
process that is  the bifractional Brownian motion $B^{H,K}$ with
parameters $H\in~(0,1)$ and $K\in(1,2)$ such that $HK\in(0,1)$. A
remarkable difference between the case $K\in(0,1)$ and our situation
is that this process is a semimartingale when $2HK=1$.
\end{abstract}

\maketitle

\section{Introduction}

Houdr\'e and Villa in \cite{HV} gave the first introduction to the
bifractional Brownian motion (bifBm)
$B^{H,K}=\left(B^{H,K}_t;t\geq0\right)$ with parameters $H\in(0,1)$
and $K\in(0,1]$  which is defined as a centered Gaussian process,
with covariance function
\begin{eqnarray*}R^{H,K}(t,s)=E\left(B^{H,K}_tB^{H,K}_s\right)
=\frac{1}{2^K}\left(\left(t^{2H}+s^{2H}\right)^K-|t-s|^{2HK}\right),
\end{eqnarray*}for every $ s,\ t\geq0$.\\
The case $K=1$  corresponds to the fractional Brownian motion (fBm)
with Hurst parameter $H$. Some properties of the bifractional
Brownian motion have been studied by Russo and Tudor in \cite{RT}.
In fact, in \cite{RT} it is shown that the bifractional Brownian
motion behaves as a fractional Brownian motion with Hurst parameter
$HK$. A stochastic calculus with respect to this process has been
recently developed by Kruk, Russo and Tudor \cite{KRT} and Es-Sebaiy
and Tudor \cite{ET}.

In this paper we prove that, with $H\in(0,1)$ and $HK\in(0,1)$, the
process $B^{H,K}$ can be extended for $1<K<2$. The case $H=\frac12$
and $1<K<2$ plays a role to give an extension of sub-fractional
Brownian motion (subfBm) (see \cite{BGT1}). The subfBm
$(\xi^h_t,t\geq0)$ with parameter $0<h\leq2$ is a centered Gaussian
process with covariance:
\[E\left(\xi^h_t\xi^h_s\right)=C_h\left(t^{2h}+s^{2h}-\frac12\left((t+s)^{2h}+|t-s|^{2h}\right)\right);\quad s,\ t\geq0\]
where $C_h=1 $ if $0<h<1$ and $C_h=2(1-h) $ if $1<h\leq2$.

\section{Definition of bifractional Brownian motion with parameter $K\in(1,2)$}

For any $K\in (0,2)$, let $X^K=(X_t^K, t\geq0)$ be a Gaussian
process defined by
\begin{equation}\label{e:david}
X_t^K=\int_0^{\infty}(1-e^{-rt})r^{-\frac{1+K}{2}}dW_r,\quad
t\geq0\end{equation} where $(W_t,t\geq0)$ is a standard Brownian
motion.

This process was introduced in \cite{LN} for $K\in(0,1)$ in order to
obtain a decomposition of the bifractional Brownian motion with
$H\in(0,1)$ and $K\in(0,1)$. More precisely, they prove the
following result:
\begin{theorem}[see \cite{LN}]\label{leinu}
Let $B^{H,K}$ a bifractional Brownian motion with parameters
$H\in(0,1)$ and $K\in(0,1)$,  $B^{HK}$ be a fractional Brownian
motion with Hurst parameter $HK\in(0,1)$ and $W=\{W_t,t\geq0\}$ a
standard Brownian motion. Let $X^K$ be the process given by
(\ref{e:david}). If we suppose that $B^{H,K}$ and $W$ are
independents, then the processes
$\{Y_t=C_1X_{t^{2H}}^K+B_t^{H,K},t\geq0\}$ and
$\{C_2B_t^{HK},t\geq0\}$ have the same distribution, where
$C_1=\sqrt{\frac{2^{-K}K}{\Gamma(1-K)}}$ and $C_2=2^{\frac{1-K}2}$.
\end{theorem}

The process defined in (\ref{e:david}) has good properties. The
following result is proved in \cite{LN} for the case $K\in(0,1)$ and
extended to the case $K\in(1,2)$ in \cite{BB} and \cite{RCT}:

\begin{proposition}[see \cite{BB},\cite{LN} and \cite{RCT}]\label{t:covX} The process
$X^{K}=\{X_t^{K},t\geq0\}$ is Gaussian, centered, and its covariance
function is:
\begin{equation}\label{e:david_cov}
\mathrm{Cov}(X_t^{K},X_s^{K})=\left\{
\begin{array}{ll}
\tfrac{\Gamma(1-K)}{K}\left[t^{K}+s^{K}-(t+s)^{K}\right] & \!\!\!\!\!\! \mbox{ if ${K}\in (0,1)$,} \\
\tfrac{\Gamma(2-K)}{K(K-1)}\left[(t+s)^{K}-t^{K}-s^{K}\right] &
\!\!\!\!\!\! \mbox{ if ${K}\in (1,2)$}.
\end{array}
\right.
\end{equation}
Moreover, $X^{K}$ has a version with trajectories which are
infinitely differentiable on $(0,\infty)$ and absolutely continuous
on $[0,\infty)$.
\end{proposition}

Using the fact that when $K\in(1,2)$, the covariance function of
$X^K$ is given by
$$Cov(X_t^K,X_s^K)=\frac{\Gamma(2-K)}{K(K-1)}\left((t+s)^K-t^K-s^K\right),$$
and considering also the process
\begin{equation}\label{e:davidmod}X^{H,K}_t=X_{t^{2H}}^K; \ t\geq0,\end{equation} we can prove the
following result:

\begin{theorem}\label{bk}
Assume $H\in(0,1)$ and $K\in(1,2)$ with $HK\in(0,1)$. Let $B^{HK}$
be a fractional Brownian motion, and $W=\{W_t,t\geq0\}$ a standard
Brownian motion. Let $X^{K,H}$ the process defined in
(\ref{e:davidmod}). If we suppose that $B^{HK}$ and $W$ are
independents, then the processes
\begin{eqnarray}\label{decomposition}B^{H,K}_t=aB^{HK}_t+bX^{H,K}_t,
\end{eqnarray}where
$a=\sqrt{2^{1-K}}$ and $b=\sqrt{\frac{K(K-1)}{2^K\Gamma(2-K)}}$ is a
centered Gaussian process with covariance function
\begin{eqnarray*}E\left(B^{H,K}_tB^{H,K}_s\right)=\frac{1}{2^K}\left(\left(t^{2H}+s^{2H}\right)^K-|t-s|^{2HK}\right);\quad s,\ t\geq0.
\end{eqnarray*}
\end{theorem}

\begin{proof}
It is obvious that the process defined in (\ref{decomposition}) is a
centered Gaussian process. On the other hand, its covariance
functions is given by
\begin{eqnarray*}
E\left(B^{H,K}_tB^{H,K}_s\right)&=&a^2
E\left(B^{HK}_tB^{HK}_s\right)+b^2E\left(X^{H,K}_tX^{H,K}_s\right)\\
&=&\frac1{2^K}\left(t^{2HK}+s^{2HK}-
|t-s|^{2HK}\right)\\&&\qquad\qquad+\frac1{2^K}\left(\left(t^{2H}+s^{2H}\right)^K-
t^{2HK}-s^{2HK}\right)\\
&=&
\frac{1}{2^K}\left(\left(t^{2H}+s^{2H}\right)^K-|t-s|^{2HK}\right),
\end{eqnarray*}
which completes the proof.
\end{proof}

Thus, the bifractional Brownian motion $B^{H,K}$  with parameters
$H\in(0,1)$ and $K\in(1,2)$ such that $HK\in(0,1)$ is well defined
and it has a decomposition as a sum of a fBm $B^{HK}$ and an
absolutely continuous process $X^{H,K}$.

\begin{remark}Assume that $2HK=1$. Russo and Tudor \cite{RT} proved that if
$K$ belong to $(0,1)$, the process $B^{H,K}$ is not a semimartingale. But in the
case when $1<K<2$,  $B^{H,K}$ is  a semimartingale because we have a
decomposition of this process as a sum of a Brownian motion
$B^{\frac{1}{2}}$ and a finite variation process $X^{H,K}$.
\end{remark}

The following decomposition is exploited to  prove the quasi-helix
property (in the sense of J.P. Kahane) of $B^{H,K}$. This result is
satisfied for all $K\in(0,2)$.
\begin{proposition}\label{prop_decomposition}Let $H\in(0,1)$ and $K\in(0,2)$ such that $HK\in(0,1)$. Let $(\xi_t^{K/2}, t\geq0)$ be a sub-fractional
 Brownian motion  with  parameter $K/2\in(0,1)$, independent to $B^{H,K}$ and suppose that $(B_t^{K/2}, t\geq0)$ and $(B_t^{HK}, t\geq0)$ are two independent fractional Brownian motions
  with Hurst parameter
  $K/2\in(0,1)$ and   $HK\in(0,1)$, respectively. We set $\xi_t^{K,H}=\xi_{t^{2H}}^{K/2}$ and $\widetilde{B}^{H,K}_t=B_{t^{2H}}^{K/2}$, $t\geq0$.  Then,
   it holds that
\begin{eqnarray}B^{H,K}+\sqrt{2^{1-K}}\xi^{K,H}\overset{(d)}{=}\sqrt{2^{1-K}}\left(\widetilde{B}^{H,K}+B^{HK}\right)
\end{eqnarray}where $\stackrel{d}{=}$ denotes that both processes have the same distribution.
\end{proposition}
\begin{proof} The result follows easily from the independence  and
the fact that their  corresponding covariance functions satisfy the
following equality for all $s,t\geq0$
\begin{eqnarray*}R^{H,K}(t,s)&=&\frac{1}{2^K}\left((t^{2H}+s^{2H})^K-|t-s|^{2HK}\right)\\&=&
2^{1-K}\left[-Cov(\xi^{K,H}_t,\xi^{K,H}_s)
+Cov(\widetilde{B}^{H,K}_t,\widetilde{B}^{H,K}_s)\right.
\\&&\qquad\qquad\qquad\qquad\qquad\qquad\left.+Cov(B_t^{HK},B_s^{HK})\right].
\end{eqnarray*}
\end{proof}

\begin{proposition}
Let $H\in(0,1)$ and $K\in(1,2)$ such that $HK\in(0,1)$. Then for any
$t,s\geq0$,\\ if $0<H\leq1/2$
$$2^{1-K}|t-s|^{2HK}\leq
E\left(B_t^{H,K}-B_s^{H,K}\right)^2\leq |t-s|^{2HK},$$ and if
$1/2<H<1$ $$2^{1-K}|t-s|^{2HK}\leq
E\left(B_t^{H,K}-B_s^{H,K}\right)^2\leq 2^{2-K}|t-s|^{2HK}.$$
\end{proposition}
\begin{proof} Using the proposition \ref{prop_decomposition}, we
obtain
\begin{eqnarray*}&&E\left(B^{H,K}_{t}-B^{H,K}_{s}\right)^2\\&=&2^{1-K}\left(-E\left(\xi^{\frac{K}{2}}_{t^{2H}}-\xi^{\frac{K}{2}}_{s^{2H}}\right)^2
+E\left(B^{\frac{K}{2}}_{t^{2H}}-B^{\frac{K}{2}}_{s^{2H}}\right)^2+E\left(B^{HK}_{t}-B^{HK}_{s}\right)^2\right)\\&=&
2^{1-K}\left(-E\left(\xi^{\frac{K}{2}}_{t^{2H}}-\xi^{\frac{K}{2}}_{s^{2H}}\right)^2
+\left|t^{2H}-s^{2H}\right|^K+\left|t-s\right|^{2HK}\right).
\end{eqnarray*}
On the other hand, from \cite{BGT} we have
\begin{eqnarray*}(2-2^{K-1})\left|t^{2H}-s^{2H}\right|^K \leq
E\left(\xi^{\frac{K}{2}}_{t^{2H}}-\xi^{\frac{K}{2}}_{s^{2H}}\right)^2\leq
\left|t^{2H}-s^{2H}\right|^K.
\end{eqnarray*}
Thus
\begin{eqnarray*}2^{1-K}\left|t-s\right|^{2HK}&\leq&
E\left(B^{H,K}_{t}-B^{H,K}_{s}\right)^2\\&\leq&
2^{1-K}\left(\left|t-s\right|^{2HK}+
(2^{K-1}-1)\left|t^{2H}-s^{2H}\right|^K\right).
\end{eqnarray*}
Then we deduce that for every $H\in(0,1)$, $K\in(1,2)$ with
$HK\in(0,1)$ \begin{eqnarray*}2^{1-K}\left|t-s\right|^{2HK}\leq
E\left(B^{H,K}_{t}-B^{H,K}_{s}\right)^2\end{eqnarray*} and the other
hand for every $H\in(0,\frac12]$, $K\in(1,2)$ we have
\begin{eqnarray*}
E\left(B^{H,K}_{t}-B^{H,K}_{s}\right)^2&\leq&
2^{1-K}\left(\left|t-s\right|^{2HK}+
(2^{K-1}-1)\left|t^{2H}-s^{2H}\right|^K\right)\\&\leq&\left|t-s\right|^{2HK}.
\end{eqnarray*}The last  inequality is satisfied from the fact that
$\left|t^{2H}-s^{2H}\right|\leq\left|t-s\right|^{2H}$ for
$H\in~(0,\frac12]$.\\
To complete the proof,  it remains to show that for every
$H\in(\frac12,1)$, $K\in(1,2)$ with $HK\in(0,1)$ (observe that in
this situation we have $HK\in(\frac12,1)$)
$$E\left(B_t^{H,K}-B_s^{H,K}\right)^2\leq 2^{2-K}|t-s|^{2HK}.$$
 Notice that,
\begin{eqnarray*}
E\left(B_t^{H,K}-B_s^{H,K}\right)^2\nonumber
&=&\frac1{2^{K}}\left[(2t^{2H})^{K}+(2s^{2H})^{K}\right.\nonumber\\&&\qquad\qquad\qquad\qquad
\left.-2\left((t^{2H}+s^{2H})^{K}-|t-s|^{2HK}\right)\right]\nonumber\\
&=&\frac2{2^K}|t-s|^{2HK}+\left(t^{2HK}+s^{2HK}-\frac2{2^K}(t^{2H}+s^{2H})^K\right).\label{conv}
\end{eqnarray*}
Hence it is enough to prove that
$$t^{2HK}+s^{2HK}-\frac2{2^K}(t^{2H}+s^{2H})^K\leq 2^{1-K}|t-s|^{2HK},$$
or equivalently
$$t^{2HK}+s^{2HK}\leq 2^{1-K}\left((t^{2H}+s^{2H})^K+|t-s|^{2HK}\right).$$
From now on we will assume, bethought loss of generality, that
$s\leq t$. Dividing by $t^{2HK}$ we obtain that we have to prove
that
$$1+\left(\frac s t\right)^{2HK}\leq 2^{1-K}\left(\left(1+\left(\frac s t\right)^{2H}\right)^K+\left(1-\frac s t\right)^{2HK}\right).$$
Equivalently we have to prove that, for any $u\in(0,1]$ the function
$$f(u):=2^{1-K}\left[(1+u^{2H})^K+(1-u)^{2HK} \right]-u^{2HK}-1$$
is positive.\\
Observe that $f(1)=0$, so, it is enough to see that the derivative
of this function is negative for $u\in(0,1]$. But,
$$f'(u)=2HK2^{1-K}u^{2HK-1}\left[\left(1+\frac1{u^{2H}}\right)^{K-1}-\left(\frac1u-1\right)^{2HK-1}-2^{K-1}\right].$$
To prove that $f'(u)\leq0$ for $u\in(0,1]$ it is enough to see that
the function
$$h(u):=\left(1+\frac1{u^{2H}}\right)^{K-1}-\left(\frac1u-1\right)^{2HK-1}-2^{K-1},$$
is negative for $u\in(0,1]$. But, since $h(1)=0$, it is enough to
prove that its derivative $h'(u)\geq0$ for $u\in(0,1]$. But,
$$
h'(u)=\frac1{u^{2HK}}\left(-2(K-1)H(u^{2H}+1)^{K-2}u^{2H-1}+(1-u)^{2HK-2}(2HK-1)\right).$$
Observe that $u^{2H-1}\leq1$ because $H\in(\frac12,1)$,
$(u^{2H}+1)^{K-2}\leq 1$ and $$(1-u)^{2HK-2}\geq1.$$ So,
$$h'(u)\geq\frac1{u^{2HK}}\left(-2(K-1)H+2HK-1\right)=\frac1{u^{2HK}}(2H-1)\geq0,$$
because $H\geq\frac12$. The prove is now complete.
\end{proof}

\begin{proposition}\label{proprieties}Suppose that $H\in(0,1)$, $K\in(1,2)$ such that $HK\in(0,1)$. The bifBm $B^{H,K}$  has the following properties
\begin{itemize}
\item[i)] $B^{H,K}$ is a self-similar process with index $HK$, i.e.
\[\left(B^{H,K}_{at},t\geq0\right)\stackrel{d}{=}\left(a^{HK}B^{H,K}_{t},t\geq0\right),\quad\mbox{ for each } a>0.\]
\item[ii)] $B^{H,K}$ has the same long-range property of the fBm $B^{HK}$,
i.e. $B^{H,K}$ has  the short-memory for $HK<\frac{1}{2}$ and it has
long-memory for $HK>\frac{1}{2}$.
\item[iii)] $B^{H,K}$ has a $\frac{1}{HK}$-variation equals to $2^{\frac{1-K}{HK}}\lambda t$ with  $\lambda=E(|N|^{\frac{1}{HK}})$ and  $N$ being a standard normal random
variable, i.e.
\begin{eqnarray*}
\sum_{j=1}^{n}\left(B^{H,K}_{{t_{j}^n}}-B^{H,K}_{{t_{j-1}^n}}\right)^{\frac{1}{HK}}\underset{n\rightarrow\infty}{\longrightarrow}
2^{\frac{1-K}{HK}}\lambda t\ \mbox{ in } L^1(\Omega).
\end{eqnarray*} where $ 0 = t_0^n < \ldots< t_n^n = t$ denotes a partition of
$[0,t]$.

\item[iv)] $B^{H,K}$ is not a semimartingale if $2HK\neq 1$.
\end{itemize}
\end{proposition}
\noindent The proof of the proposition \ref{proprieties} is
straightforward from \cite{RT} and \cite{ET}.

\section{\bf{Space of integrable functions with respect to bifractional
Brownian motion}}

Let us consider $\mathcal E$ the set of simple functions on $[0,T]$.
Generally, if $U:=(U_t,\,t\in[0,T])$ is a continuous, centered
Gaussian process, we denote by $\mathcal H_U$ the Hilbert space
defined as the closure of $\mathcal E$ with respect to the scalar
product
$$\left<{\bf 1}_{[0,t]},{\bf 1}_{[0,s]}\right>_{\mathcal H}=E\left(U_tU_s\right).$$

In the case of the standard Brownian motion $W$, the space $\mathcal
H_W$ is $L^2([0,T])$. On the other hand, for the fractional Brownian
motion $B^H$, the space $\mathcal H_{B^H}$ is the set of
restrictions to the space of test functions $\mathcal D((0,T))$ of
the distributions of $W^{\frac12-H,2}(\mathbb{R})$ with support
contained in $[0,T]$ (see \cite{J}). In the case $H\in(0,\frac12)$
all the elements of the domain are functions, and the space
$\mathcal H_{B^H}$ coincides with the fractional Sobolev space
$I_{0^{+}}^{\frac12-H}(L^2([0,T]))$ (see for instance \cite{DU}),
but in the case $H\in(\frac12,1)$ this space contains distributions
which are not given by any function.

As a direct consequence of Theorem \ref{bk} we have the following
relation between $\mathcal H_{B^H}$, $\mathcal H_{B^{H,K}}$ and
$\mathcal H_{X^{H,K}}$, where $B^{H,K}$ is the bifractional Brownian
motion and $X^{H,K}$ is the process defined in (\ref{e:davidmod}).

\begin{proposition}
Let $H\in(0,1)$ and $K\in(1,2)$ with  $HK\in(0,1)$. Then it holds
that
$$\mathcal H_{X^{H,K}}\cap\mathcal H_{B^{HK}}=\mathcal H_{B^{H,K}} $$
\end{proposition}

If we consider the processes appearing in Proposition
\ref{prop_decomposition} we have also the following result:
\begin{proposition}
Let $H\in(0,1)$. For every  $K\in(0,2)$ with  $HK\in(0,1)$ the
following equality holds
$$ \mathcal H_{\xi^{H,K}}\cap\mathcal H_{B^{H,K}}=\mathcal
H_{\widetilde{B}^{H,K}}\cap\mathcal H_{B^{HK}}.$$
\end{proposition}

\begin{proof}
Both propositions are a direct consequence of the two decompositions
into the sum of two independent processes proved in Theorem \ref{bk}
and Proposition \ref{prop_decomposition}.
\end{proof}
\begin{remark} For the case $K\in(0,1)$ we have the following equality (see \cite{LN})
$$\mathcal H_{B^{HK}}=\mathcal H_{X^{H,K}}\cap\mathcal H_{B^{H,K}}.$$
\end{remark}

\section{Weak convergence towards the bifractional Brownian
motion}

Another direct consequence of the decomposition for the bifractional
Brownian motion with $H\in(0,1)$, $K\in(1,2)$ and $HK\in(0,1)$ is
the following result of convergence in law in the space $\mathcal
C([0,T])$.

Recall that the fractional Brownian motion of Hurst parameter $H\in
(0,1)$ admits an integral representation of the form (see for
instance \cite{AMN})
$$
B_t^{H}=\int_0^t K^{H}(t,s)dW_s,
$$
where $W$ is a standard Brownian motion and the kernel $K^{H}$ is
defined on the set $\{0<s<t\}$
 and given by
\begin{equation}\label{eq2}
K^{H}(t,s)=d_{H} (t-s)^{H-\frac12}+ d_{H}(\frac12-H)\int_s^t
(u-s)^{H-\frac32}\left(1-\left(\frac{s}{u}\right)^{\frac12-H}\right)
du,
\end{equation} with $d_{H}$ the following normalizing constant
 $$ d_{H}=\left(\frac{2H\Gamma(\frac32-H)}{\Gamma(H+\frac12)
 \Gamma(2-2H)}\right)^{\frac12}.$$

\begin{theorem}
Let $H\in(0,1)$ and $K\in(1,2)$ with $HK\in(0,1)$. Consider
$\theta\in(0,\pi)\cup(\pi,2\pi)$ such that if $HK\in(0,\frac{1}{4}]$
then $\theta$ satisfies that $\cos((2i+1)\theta)\neq1$ for all
$i\in\mathbb{N}$ such that
$i\leq\frac{1}{4}\left[\frac{1}{H}\right]$. Set $a=\sqrt{2^{1-K}}$
and $b=\sqrt{\frac{K(K-1)}{2^K\Gamma(2-K)}}$. Define the processes,
\begin{eqnarray*}
{B}^{HK}_{\epsilon}&=&\left\{\frac{2}{\epsilon}\int_0^TK^{HK}(t,s)\sin\left(\theta
N_{\frac{2s}{\epsilon^2}}\right)\mbox{d}s,\quad t\in[0,T]\right\},\\
X^{H,K}_{\epsilon}&=&\left\{\frac{2}{\epsilon}\int_0^\infty(1-e^{-st^{2H}})s^{-\frac{1+K}{2}}\cos\left(\theta
N_{\frac{2s}{\epsilon^2}}\right)\mbox{d}s,\quad t\in[0,T]\right\},
\end{eqnarray*}
where $K^{HK}(t,s)$ is the kernel defined in (\ref{eq2}). Then,
$$\{Y^H_\epsilon(t)=aB^{HK}_\epsilon(t)+bX^{H,K}_\epsilon(t),
t\in[0,T]\}$$ weakly converges in $\mathcal C([0,T])$ to a
bifractional Brownian motion.
\end{theorem}

\begin{proof}
Applying Theorems 3.2 and 3.5 of \cite{BB} we know that,
respectively, the processes  ${B}^{HK}_\epsilon$ and
$X_\epsilon^{H,K}$ converge in law in $\mathcal C([0,T])$ towards a
fBm  $B^{HK}$ and to the process $X^{H,K}$. Moreover, applying
Theorem 2.1 of \cite{BB}, we know that the limit laws are
independent. Hence, we are under the hypothesis of the decomposition
obtained in Theorem \ref{bk}, which proves the stated result.
\end{proof}

\begin{remark}
Obviously we can also obtain the same result interchanging the roles
of the sinus and the cosinus functions in the definition of the
approximating processes.
\end{remark}

\par\bigskip\noindent
{\bf Acknowledgment.}  The authors would like to thank the editor Hui-Hsiung Kuo  and  referees for the valuable comments.

\bibliographystyle{amsplain}

\end{document}